\documentclass{article}

\usepackage{amssymb}
\usepackage{latexsym}
\usepackage{amsmath}
\usepackage{mathrsfs}
\usepackage[all]{xy}
\usepackage{enumerate}
\usepackage{amsthm}
\usepackage[dvips]{graphicx}
\usepackage{psfrag}
\usepackage{ifthen}
\usepackage{url}

\newtheorem{thm}{Theorem}[section]
\newtheorem{lem}[thm]{Lemma}
\newtheorem{prop}[thm]{Proposition}
\newtheorem{cor}[thm]{Corollary}
\theoremstyle{definition}
\newtheorem{defn}[thm]{Definition}
\newtheorem{obs}{Observation}

\newtheorem{conv}[thm]{Convention}

\newcommand{\insertimage}[2]
	   {\includegraphics[width=#2\textwidth]{#1}}

\newcommand{\define}[1]{\emph{#1}}

\newcommand{\mpo}{{\mp 1}}
\newcommand{\mo}{{-1}}
\newcommand{\pmo}{{\pm 1}}

\newcommand{\Z}{\mathbb{Z}}
\newcommand{\bra}{\langle}
\newcommand{\kett}{\rangle}
\newcommand{\ol}{\overline}

\newcommand{\N}{\mathbb{N}}

\newcommand{\bk}[1]{{\bra #1 \kett}}
\newcommand{\bel}[1]{\begin{equation}\label{#1}}
\newcommand{\be}{\begin{equation}}
\newcommand{\ee}{\end{equation}}

\newcommand{\tr}{\textrm}
\newcommand{\wt}{\widetilde}

\newcommand{\fix}[2]{\textrm{Fix}_{#1}(#2)}

\newcommand{\twocomplx}[1]{\mathscr{#1}}
\newcommand{\bouquet}[1]{\mathcal{B}_{#1}}
\newcommand{\graphofgroups}[1]{\mathcal{G}({#1})}
\newcommand{\diam}[1]{\textrm{diameter}(#1)}
\newcommand{\per}[2]{\textrm{period}_{#1}(#2)}
\newcommand{\odiag}[1]{\mathcal{D}_{#1}}

\renewcommand{\H}{\mathbb{H}}
\newcommand{\T}{T_{\graphofgroups{\Gamma}}}
\newcommand{\func}{\mathfrak{n}}
\newcommand{\refine}[2]{\twocomplx{#1}(\odiag{#2})}
\newcommand{\res}[1]{\rho_{\odiag{f}}}
\newcommand{\alleqns}[3]{\mathcal{E}(#1,#2,#3)}

\title{Bulitko's Lemma for acylindrical splittings}
\author{Nicholas W.M. Touikan\footnote{Supported by NSERC PDF}\\Centre Interuniversitaire de Recherche en G\'eom\'etrie et Topologie \\ Universit\'e du Qu\'ebec \`a Montr\'eal \\email:
  \texttt{nicholas.touikan@gmail.com}}

\begin{document}
\maketitle
\abstract{We generalize Bulitko's Lemma to equations over (or
  homomorphisms into) groups that have $\kappa$-acylindrical
  splittings. This is a key technical lemma used in the author's
  algorithm to find tracks in 2-complexes.}
\section{Introduction}
In \cite{Bulitko} Bulitko proved what is now known as Bulitko's Lemma which we restate:

\begin{thm}[Bulitko's Lemma] For a system of equations $S(X,A)$ over a
  free group $F(A)$, there exists a computable number $n(|S(X,A)|)$
  such that for any cyclically reduced word $p \in F(A)$ and any
  solution $\phi$ of $S(X,A)$ there exists another solution $\phi^*$
  such that for each $x \in X$ every maximal subword of the form $p^m$
  in the freely reduced word corresponding to $\phi^*(x)$ has
  periodicity $|m|$ bounded above by $n(|S(X,A|)$.
  \end{thm}	
  This result is also known to hold for equations or homomorphisms into
  free products. An important application of this lemma is in
  Makanin's algorithm to solve systems of equations over free
  groups. We will generalize this result by replacing the free group $F(A)$
  above by a group $\H$ that splits as the fundamental group of a
  $\kappa$-acylindrical graph of groups $\graphofgroups{\Gamma}$ and
  where the element $p$ is replaced by some arbitrary hyperbolic
  element of $\H$ w.r.t. $\graphofgroups{\Gamma}$.

  Our main motivation is a key technical lemma in
  \cite{touikan-grushko}. Specifically, in order to perform an
  \emph{unmeasured periodic merger} we need an acylindrical Bulitko
  lemma. Because the proof is rather long (though not really deep) and
  may be of independent interest, it seems worthwhile to present the
  Acylindrical Bulitko Lemma as a separate result.

  Throughout this paper $\H$ shall designate a fixed group that is the
  fundamental group of a graph of groups $\graphofgroups{\Gamma}$ with
  underlying graph $\Gamma$. We shall call the decomposition of $\H$
  as the fundamental group of a graph of groups a
  \define{splitting}. Let $\T$ denote the corresponding Bass-Serre
  tree. We shall consider the action:
\[\left.\begin{array}{c} \H \times \T \rightarrow \T \\ (h,x) \mapsto
    h\cdot x \end{array}\right.\] to be implicit in the definition of $\T$. We moreover consider $\T$ to be a metric space with the combinatorial metric. For an arc $\rho \subset \T$ we denote its length by $|\rho|$.

\begin{defn} We will say the the splitting of $\H$ is
  \define{$\kappa$-acylindrical} if for all $h \in \H$
  \[\diam{\fix{\T}{h}} \leq \kappa\] where $\fix{\T}{h}$ denotes the subset
  of $\T$ that is fixed pointwise by the action of $h$. We will also
  say that \define{$\T$ is $\kappa$-acylindrical}.
\end{defn}

The goal of this paper is to generalize Bulitko's Lemma to equations over
groups that have $\kappa$-acylindrical splittings. The study of equations over groups is equivalent to the study of homomorphisms into groups, we consider the correspondence to be obvious and will use this latter approach throughout the paper.

\subsection{Statement of the Main Theorem}

In order to state the main result precisely we need some more
terminology. Let $p$ be some hyperbolic element of $\H$ (w.r.t. the
given splitting of $\H$) and let $\lambda\subset \T$ denote its axis.
Consider the set of segments \[\big\{\lambda' \subset \lambda \mid \lambda = \bigcup_{n
  \in \Z} p^n \lambda'\big\}. \] A minimal element of this set
w.r.t. inclusion is called a \define{fundamental domain} of $\lambda$.

For an element $h \in \H$, let $[v_0,h\cdot v_0]$ denote the geodesic
between $v_0$ and $h\cdot v_0$. Let $g \in \H$, if a segment
\[\sigma = [v_0,h\cdot v_0] \cap g\cdot \lambda\] is non-empty then we
call it an $\lambda$-periodic subsegment of $[v_0,h\cdot v_0]$. The
\define{$\lambda$-periodicity} of $\sigma$ is the integer\[ \lfloor
\frac{|\sigma|}{|\lambda_0|}\rfloor\] where $\lambda_0$ is a fundamental domain of
$\lambda$. We denote this $\per{l}{\sigma}$. We can now precisely state the
main theorem

\begin{thm}[Acylindrical Bulitko Lemma]\label{thm:main}
  There exists a computable function $\func:\N \times \N \times
  \N \rightarrow \N$ such that for any nontrivial homomorphism $\phi:
  G \rightarrow \H$; where the group $G$ has a finite
  presentation $\bk{Y \mid S}$ and the group $\H$ has a
  $\kappa$-acylindrical splitting with based Bass-Serre tree
  $(\T,t_0)$; and for any hyperbolic element in $p\in \mathbb{H}$
  (denote its axis $\lambda \subset \T)$, there exists a homomorphism
  $\phi^*: G \rightarrow \H$ such that for all $y \in Y$
  \begin{itemize}
  \item if $[t_0,\phi(y)\cdot t_o]$ has no $\lambda$-periodic subsegments, then
    $\phi(y)=\phi^*(y)$, and
  \item if $[t_0,\phi(y)\cdot t_o]$ has $\lambda$-periodic
    subsegments, then there is a bijective correspondence between the
    $\lambda$-periodic subsegments of $[t_0,\phi(y)\cdot t_o]$ and
    $[t_0,\phi^*(y)\cdot t_o]$, but the $\lambda$-periodicity of all
    the $\lambda$-periodic subsegments of $[t_0,\phi^*(y)\cdot t_o]$
    is at most $\func(|Y|,|S|,\kappa)$.
  \item for all $y \in Y$ the length of $[t_0,\phi^*(y)\cdot t_o]$ is
    at most the length of $[t_0,\phi(y)\cdot t_o]$.
  \end{itemize}
\end{thm}

By considering the free group as a multiple independent HNN extension
of the trivial group we recover Bulitko's Lemma. This next Corollary
may be interesting for the study of the universal theory of graphs of
groups in terms of their vertex groups. It roughly states that if a
system of equations and inequations has a solution, then it has a
solution of bounded periodicity.

\begin{cor}\label{cor:preservation}
  Let everything be as in statement of Theorem \ref{thm:main} and let
  $F \subset G$ be a finite set of words in $Y^\pmo$, then there is
  another computable function $\func'$ which takes as input
  $(F,|Y|,|S|,\kappa)$ and outputs a natural number which satisfies
  the same properties as $\func$ as stated in Theorem \ref{thm:main}
  but with the added property that if $f \in F$ was not mapped to
  1 via $\phi$ then it's not mapped to 1 via $\phi^*$.
\end{cor}

\begin{proof}
  Change the presentation $\bk{Y \mid S}$ of $G$ so that
  elements of $F$ lie in the generating set and apply Theorem \ref{thm:main}.
\end{proof}

\begin{conv}
  For the rest of the paper we will assume that our splittings only
  have one edge. This causes no loss of generality since we can always
  collapse our splitting so that the underlying graph only has one
  edge while keeping the element $p \in \H$ in Theorem \ref{thm:main}
  hyperbolic. This new splitting will be $\kappa'$-acylindrical where
  $\kappa' \leq \kappa$.
\end{conv}
\section{Presentations representing splittings and $E$-cells.}
\label{sec:e-cells}

Let $\H$ be a f.g. group which splits either as an HNN extension or as a
free product with amalgamation, we will associate to $\H$ a
presentation representing the splitting.
\begin{itemize}
\item If $\H = \bk{A,t \mid t^\mo c_i t = d_i; i \in I}$, i.e. $\H$ is an HNN
  extension of $A$ where $C=\bk{c_i}_{i \in I}$ and $D=\bk{d_i}_{i \in
    I}$ represent generating sets of the associated subgroups, and the
  map $c_i \mapsto d_i$ induces an isomorphism from $C$ to $D$, then we
  can assume that $\{c_i,d_i\}_{i \in I} \subset X$ and
  write \begin{equation}\label{eqn:hnn} \H = \langle X,t \mid R, E
    \rangle\end{equation} where $A = \bk{X \mid R}$ and $E = \{t^\mo
  c_i t = d_i \mid i \in I\}$.
\item If $\H = A*_c B$, we let $\bk{a_i}_{i\in I}$ and $\bk{b_i}_{i \in I}$ represent
  generating sets of $C$ in $A$ and $B$ resp. where the identification
  is given by $a_i = b_i; i \in I$ We assume that $\{a_i\}_{i \in I}
\subset X$ and $\{b_i\}_{i \in I} \subset U$ and we
write \begin{equation}\label{eqn:amalgam}\H = \bra X,U \mid R,T,E
  \kett\end{equation} where $A = \bk{X \mid R}$, $B =\bk{U \mid T},$
and $E = \{a_i = b_i \mid  i \in I\}$.
\end{itemize}

We now construct 2-complexes corresponding to the presentations
(\ref{eqn:hnn}),(\ref{eqn:amalgam}). For a set $S$ we will let
$\bouquet{S}$ be the bouquet of circles where the circles are oriented
and labeled by elements in $S$. We refer the reader to
\cite{Scott-Wall} for a discussion on the connection between graphs of
spaces and graphs of groups, although we note that our construction
differs from theirs.
\begin{itemize}
\item In case we have an HNN extension, i.e.  (\ref{eqn:hnn}), we
  start with the 2-complex $\twocomplx{A}$ corresponding to the
  presentation $\bk{X \mid R}$ and remember that we are assuming that
  the generating sets $\{c_i\},\{d_i\}$ of the associated subgroups
  are included in $X$, we therefore assume that $\bouquet{\{c_i\}}
  \subset \bouquet{X} \supset \bouquet{\{d_i\}}$ in the obvious
  way. We attach the 2-cells corresponding to elements of $R$ to
  $\bouquet{X}$ in the standard way.

  We now attach another oriented circle labelled $e$ to $\bouquet{X}$ to get
  $\bouquet{X,e}$ and for each relation $t^\mo c_i t = d_i$ in $E$
  we attach a square whose boundary is attached to $\bouquet{X,e}$
  along the path $e^\mo * c_i * e * d_i^\mo$ (we are obviously abusing
  notation.) \define{We call these squares $E$-cells}.

  We call the resulting complex $\twocomplx{H}$. It is easy to see that
  $\pi_1(\twocomplx{H}) = \H$. We also note that $\twocomplx{H}$ has a
  natural graph of spaces structure: the vertex space consists of
  $\Gamma_X$ and the 2-cells corresponding to $R$ and the edge space
  consists of the union of the $E$-cells.

\item In case were $\H$ is a free product with amalgamation, i.e.
(\ref{eqn:amalgam}), we start with disjoint cell complexes $
\twocomplx{A}, \twocomplx{B}$ corresponding to the presentations $A = \bk{X 
\mid 
R}$, $B =\bk{U \mid T}$ resp. These cell complexes have a single 
vertex. We connect these vertices by an edge labeled $e$ and for each 
relation $a_i=b_i$ in $E$ we attach squares whose boundary is attached to $
\bouquet{X}\cup \bouquet{Y} \cup e$ along the path $a_i e b_i^\mo e^\mo$. 
\define{Again we call these squares $E$-cells.}

The resulting cell complex, which we call $\twocomplx{H}$, again has a graph 
of spaces structure: the two vertex spaces are the the subcomplexes $
\twocomplx{A}, \twocomplx{B}$ and the edge space is the union of the the $E
$-cells. Again we see that $\pi_1(\twocomplx{H}) = \H.$
 \end{itemize}
 
 We call these 2-complexes the \emph{special 2-complexes corresponding
   to the presentations (\ref{eqn:hnn}) or (\ref{eqn:amalgam}).}

\begin{conv}\label{conv:complexes}
  For any f.p. group $G = \bk{Y \mid S}$ we shall use the script
  letter $\twocomplx{G}$ to denote the 2-complex associated to the
  presentation 2-complex, in particular this 2-complex has a unique
  0-cell. We shall also assume for the rest of the paper that $\H$ is
  $\pi_1(\twocomplx {H},x_0)$ where $\twocomplx{H}$ is a
  \emph{special} 2-complex as given above.  In both cases we have a
  subcomplex $\twocomplx{A} \subset \twocomplx{H}$ and the basepoint
  $x_0$ of $\twocomplx {H}$ will be the unique 0-cell in
  $\twocomplx{A}$. We will identify the subgroup $A \leq \H$ that
  appears in the presentations (\ref{eqn:hnn}) or (\ref{eqn:amalgam})
  with $\pi_1(\twocomplx {A},x_0) \leq \pi_1(\twocomplx{H},x_0)$
moreover we shall assume that the basepoint $v_0$ of $\T$ is the
vertex fixed by $A$.
\end{conv}

\section{Ol'shanskii diagrams}\label{sec:o-diags}
In \cite{Olshanskii-1989} Ol'shanskii presented a remarkably
simple yet extremely useful construction: that a homomorphism
diagram. Let $\twocomplx{H}$ be a 2-complex, besides the 2-cells in
$\twocomplx{H}^{(2)}$ for each edge $x \in \twocomplx{H}^{(1)}$ we
also consider so-called \define{singular cells} as depicted in Figure
\ref{fig:null-cells}. Singular cells were called \emph{0-cells} in
\cite{Olshanskii-1989}. The technical advantage of singular cells is
that they enable us to have non-singular Van-Kampen diagrams.

\begin{figure}
\centering
\psfrag{x}{$x$}
\psfrag{1}{1}
\insertimage{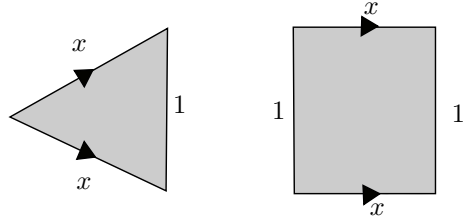}{0.5}
\caption{Singular cells}
\label{fig:null-cells}
\end{figure}

\begin{defn}
  Let $\twocomplx{G}$ and $\twocomplx{H}$ be 2-complexes. An
  \define{Ol'shanskii diagram of $\twocomplx{G}$ over $\twocomplx{H}$}
  is a filling of the 2-cells of $\twocomplx{G}$ by Van-Kampen
  diagrams over $\twocomplx{H}$ such that the induced combinatorial
  mappings from $ \twocomplx{G}^{(1)}$ to $\twocomplx{H}^{(1)}$ are
  well defined.
\end{defn}

\begin{lem}\label{lem:o-diags}
  Let $G$ and $H$ be groups which are the fundamental groups of
  2-complexes $\twocomplx{G} $ and $\twocomplx{H}$ resp. Then any
  homomorphism $f:G\rightarrow H$ induces an Ol'shanskii diagram of $
  \twocomplx{G}$ over $\twocomplx{H}$, and conversely any such diagram
  gives a homomorphism from $G$ to $H$.
\end{lem}

\begin{proof}
Let $f:G\rightarrow H$ be a homomorphism where $G$ and $H$ are the
fundamental groups of 2-complexes $\twocomplx{G}$ and
$\twocomplx{H}$ resp. From $f$ we can get a combinatorial map from
$\twocomplx{G}^{(1)}$ to $\twocomplx{H}^{(1)}$. By definition of a
homomorphism, the image of the path corresponding to the attaching map
of the boundary of each 2-cell (i.e. the relations) must be sent to
nullhomotopic loops, it follows that by using singular-cells we can fill
the 2-cells of $\twocomplx{G}$ by non-singular Van-Kampen diagrams
over $\twocomplx{H}$.

The converse is also clear from the definition.
\end{proof}

\begin{obs}
  $\odiag{f}$ in fact gives us a combinatorial map from
  $\twocomplx{G}$ to $\twocomplx{H}$.
\end{obs}

\begin{defn}\label{defn:refinement}
  We denote by $\odiag{f}$ an Ol'shanskii diagram induced from a
  homomorphism $f:\pi_1(\twocomplx{G},y_0) \rightarrow
  \pi_1(\twocomplx{H},x_0)$. $\odiag{f}$ gives a \define{refinement}
  of the CW-structure of $\twocomplx{G}$ in the sense that 1-cells get
  subdivided into edge paths and 2-cells get subdivided into
  non-singular Van-Kampen diagram. We denote this refined CW structure
  on $\twocomplx{G}$ as $\refine{G}{f}$.
\end{defn}

We note that there is a well defined map from edge paths in
$\refine{G}{f}$ to edge paths in $\twocomplx{H}$. In the case where
$\twocomplx{H}$ has a single vertex a combinatorial arc in
$\refine{G}{f}$ will always map onto a loop in $\twocomplx{H}$. This
will not always happen in our applications.

\section{Resolutions}\label{sec:e-tracks}
Let $G = \bk{Y \mid R}$ be some groups with associated presentation
2-complex $\twocomplx{G}$. Let $f:G \rightarrow \H$ be a
homomorphism. Then there is an induced Ol'shanskii diagram $\odiag{f}$
of $\twocomplx{G}$ over $\twocomplx{H}$.

The 2-complex $\twocomplx{H}$ we chose is a graph of spaces and it
follows that the universal cover $\wt{\twocomplx{H}}$ also decomposes
as graph of spaces, where the vertex spaces are lifts of \[
\tr{Closure}\big( \twocomplx{H} \setminus (\bigcup_{i\in I}
E_i)\big)\] where $E_i; i \in I$ are the $E$-cells. If we define an
equivalence relation $\sim$ on $\wt{\twocomplx{H}}$ which identifies
the vertex spaces to points and the edge spaces (these are connected
components of the lifts of the union of $E$-cells) to arcs in the
standard way we get an $\H$ equivariant isomorphism of
trees \[\xymatrix{\wt{\twocomplx{H}}/\sim \ar[r]^{\approx} & \T}\] If
$\wt{x_0}$ was a basepoint of $\wt{\twocomplx{H}}$ then it lied in a
vertex space $\wt{\twocomplx{A}}_0$ and we set the basepoint $v_0 \in
\T$ to be the vertex corresponding to this vertex space. This is
consistent with Convention \ref{conv:complexes}. We refer the reader
to \cite{Scott-Wall} for further details.

  So now $\twocomplx{G}$ is covered by 2-cells from $\twocomplx{H}$
  and this induces continuous maps:\[ \wt{(\twocomplx{G}},\wt{y_0})
  \rightarrow (\wt{\twocomplx{H}},\wt{x_0}) \rightarrow (\T,v_0).\] We
  we denote the composition \[\res{f}: \wt{(\twocomplx{G}},\wt{y_0})
  \rightarrow(\T,v_0)\] and call it a \define{resolution of $\T$ through
  $f$}. Let $g \in G$ then for example we have \[ \res{f}(g \cdot
  \wt{y_0}) = f(g)\cdot v_0.\]

  Recall that for a homomorphism $f:G \rightarrow \H$ then an
  Ol'shanskii diagram first depends on a choice of edge paths in
  $\twocomplx{H}$ corresponding to the image of each generator, and
  then a corresponding Van-Kampen diagram corresponding
  to the edge path read around the boundary of each 2-cell. 
  
\begin{defn}
  The union of the $E$-cells in inside a 2-cell $D$ of $\twocomplx{G}$ form a
  collection of \define{strips}, i.e. Cartesian products of 1-manifolds
  and an interval. \define{We shall call these $E$-strips}.
\end{defn}

\begin{defn} Let $f:G\rightarrow \H$ be a homomorphism, then an
  Ol'shanskii diagram $D(f)$ is $E$-tight if \begin{enumerate}
  \item for each edge, labeled $y_i$, in $\twocomplx{G}$ the
    corresponding edge path for $f(y_i)$ has the minimal possible
    number of $e$-labeled edges; and
  \item the Van-Kampen diagrams for each 2-cell have a minimal
    possible number of $E$-strips.
\end{enumerate}
\end{defn}

\begin{lem}\label{lem:e-tight}
  Let everything be as above. Let $D$ be a 3-sided 2-cell of $\twocomplx{G}$,
  let $\odiag{f}$ be $E$-tight and let $D'$ be the closure of the
  complement\[ \tr{closure} \Big(D \setminus \big( \bigcup_{e\tr{-cells} \subset
    \odiag{f}} E_j \big)\Big).\] If $\wt{D}$ is a lift of $D$ in
  $\twocomplx{G}$, the universal cover, and $\wt{D'} \subset \wt{D}$ the subset
  corresponding to $D'$, then each connected component of $\wt{D'}$ maps
  onto a different vertex of $\T$ via the resolution $\res{f}$.
\end{lem}
\begin{proof}
  We first note that there can be no annular $E$-strips lying in
  $\wt{D}$, since otherwise we can cut it out and fill it in with a
  Van-Kampen diagram over without any $E$-cells, contradicting $E$-tightness. Suppose we can read $y_iy_jy_k$ around the boundary
  of $D$.
  
  {\bf Claim:} No strip can intersect the same side of $\wt{D}$ in two connected components. Suppose it were possible, let $s$ be an innermost $E$-strip with this property and suppose it intersects the edge labeled $y_k$ twice. Then by Figure \ref{fig:same-edge-strip} we see that the subpath \[e^\pmo h e^\mpo\] in the path in $\twocomplx{H}$ we chose to represent $f(y_k)$ can be replaced by some path $g$ not containing any $e$-edges. This contradicts $E$-tightness and proves the claim.
  
\begin{figure}
\centering
\psfrag{yi}{$y_i$}
\psfrag{yj}{$y_j$}
\psfrag{ai}{$h$}
\psfrag{bi}{$g$}
\psfrag{ep}{$e^\pmo$}
\psfrag{em}{$e^\pmo$}
\psfrag{s}{$s$}
\insertimage{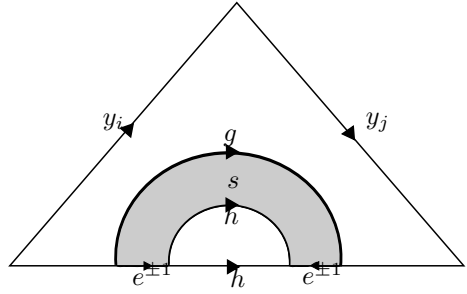}{0.50}
\caption{An $E$-strip $s$ intersecting a side of $D$ twice.}
\label{fig:same-edge-strip}
\end{figure}

  It therefore follows that each component of $\wt{D'}$ intersects the
  boundary $\partial\wt{D}$. Now suppose that two points $p,q$ on
  $\partial\wt{D}$ were mapped to the same vertex but that they lied
  different connected components of $\wt{D'}$. Then w.l.o.g. we may assume that the edge path $\gamma$ between $p,q$ in $\wt{\refine{G}{f}} \cap \partial\wt{D}$ that maps via $\res{f}$ to a path in $\T$. This path starts in a vertex $u \in \T$ goes through and edge $a \subset \T$ and then backtracks to $u$. It therefore follows that $\gamma$ can be arranged to have label $e^\pmo h e^\mpo$ where $h$ is a path in $\twocomplx{A}\cup \twocomplx{B}$ and such that the loop $e^\pmo h e^\mpo$ in $\twocomplx{H}$ is homotopic (rel. endpoint) to some loop $g \in \twocomplx{A}\cup \twocomplx{B}$. Thus if $p,q$ lied in the same side (say the one labeled $y_k$) of $\wt{D}$ then we could find another path representative for $f(y_i)$ with fewer $e$-labeled edges contradicting $E$-tightness.
  
  It therefore follows that $p$ and $q$ must lie on different sides of $\wt{D}$, but by the Claim proved earlier the $E$-strips touching $p,q$ must be as in Figure \ref{fig:different-edge-strip}, in which case we have $p'$ and $q'$ lying in the same side of $\wt{D}$, but mapping to the same vertex of $\T$ and by the previous paragraph this is impossible.
  
  \begin{figure}
  \centering
  	\psfrag{p}{$p$}
	\psfrag{q}{$q$}
	\psfrag{p'}{$p'$}
	\psfrag{q'}{$q'$}
	\insertimage{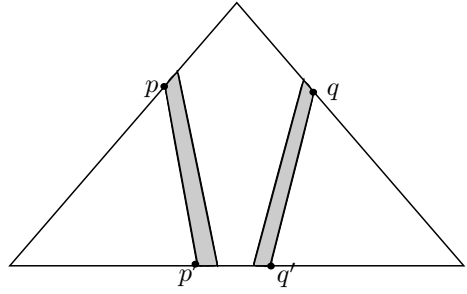}{0.50}
	\caption{An $E$-strip cannot intersect the same side of $\wt{D}$ in two connected components.}
	\label{fig:different-edge-strip} 
  \end{figure}

\end{proof}


\section{Periodic decompositions of paths}\label{sec:path-decompositions}

\begin{defn} Let $\T$ be the Bass-Serre tree as defined and let
  $\lambda \subset \T$ be a bi-infinite line. We say that $\lambda$ is
  a \emph{periodic line} if there is some $p \in H$ such that $p\cdot
  \lambda = \lambda$. If $p$ is not a proper power we say that $p$ is
  a \emph{period of $\lambda$}. \end{defn}

\begin{lem}\label{lem:periodic-intersection} Let $T$ be
  $\kappa$-acylindrical, let $\lambda$ be a periodic line and let
  $\lambda'$ be a distinct $\H$-translate of $\lambda$, and let
  $\sigma = \lambda \cap \lambda'$. Then $\per{l}{\sigma} \leq
  \kappa+3$.
\end{lem}

\begin{proof}
  W.l.o.g. we can assume that a fundamental domain $\lambda_0$ of $\lambda$ is
  coinitial with $\sigma$.
	
  Suppose towards a contradiction that the $\lambda$-periodicity of
  $\sigma$ is greater than $\kappa+3$. Let $p,p'$ be periods of
  $\lambda,\lambda'$ resp. By assumption we have that the segment $\lambda_0
  \cup p \lambda_0 \cup \cdots \cup p^{\kappa+3} \lambda_0$ lies in
  $\sigma$ and that the segment $\sigma' = \lambda_0 \cup p \lambda_0
  \cup \cdots \cup p^{\kappa+1} \lambda_0$ has length more than
  $\kappa$.  Since $p$ and $p'$ are conjugate (and therefore have same
  translation lengths) and that $ \sigma$ lies in both of their axes
  of translation we have that the segments \[p \sigma', (p' p
  )\sigma', (p^\mo p' p )\sigma', (p'^\mo p^\mo p' p)\sigma'\] all lie
  in $\sigma$ and in particular that $(p'^\mo p^\mo p' p)\sigma' =
  \sigma'$ on the other hand $p'^\mo p^\mo p' p \neq 1$ since $p$ and
  $p'$ have distinct axes (and hence do not commute). This contradicts
  $\kappa$-acylindricity. \end{proof}

\begin{cor}\label{cor:segment-uniqueness} Let $\T$ be a Bass-Serre tree
  and let $\lambda$ be a periodic line. Let $\sigma$ be some segment
  that lies in some $\H$-translate $\lambda'$ of $\lambda$. If
  $\per{l}{\sigma} > \kappa+3$ then $\sigma$ cannot be contained in
  any other $\H$-translate of $\lambda$.
\end{cor}

\begin{defn}For a fundamental domain $\lambda_0 \subset \lambda$ we will say that a periodic segment is
  \define{$(\lambda,\lambda_0)$-normal} if it is a union of translates
  of $\lambda_0$ of the form \[ \bigcup_{n=r}^s h p^n h^\mo
  \lambda_0\] for some $h \in \H$ and integers $r\leq s$.
\end{defn}

We now take $\rho$ and look at all its maximal
$(\lambda,\lambda_0)$-normal periodic subsegments. These subsegments may not
be disjoint, however by Corollary \ref{cor:segment-uniqueness} their
intersection will have $\lambda$-periodicity at most $\kappa+3$. An
$(\lambda,\lambda_0)$-normal periodic segment is \emph{long} if it has
$\lambda$-periodicity more than $2(\kappa+4)$.

\begin{defn} Let $\lambda$ and $\lambda_0$ be as above, let $\rho$ be
  a geodesic in $\T$ and let $\sigma \subset \rho$ be a long maximal
  $(\lambda,\lambda_0)$-normal periodic segment. Let
\[\sigma = \sigma_+ *\ol{\sigma}* \sigma_-\] were $*$ denotes
concatenation and $\sigma^+, \sigma^-$ both have $\lambda$-periodicity
$\kappa+4$. Then we call $\ol{\sigma}$ the
\define{$(\lambda,\lambda_0)$-stable periodic core of
  $\sigma$.} \end{defn}

\begin{defn}\label{defn:periodic-decomposition} Let $\lambda,\lambda_0$ be as
  above and let $\rho$ be a geodesic $\T$. A decomposition of
  $\rho$ as \[\rho = \rho_1 * c_1 *
  \rho_2 *c_2 *\cdots *c_n *\rho_{n+1}\] where each $c_i$ is an
  $(\lambda,\lambda_0)$-stable periodic core and each long $\lambda$-periodic subsegment of
  $\rho$ contains a $c_i$ is called an \define{$(\lambda,\lambda_0)$-periodic
    decomposition}. \end{defn}
    
\begin{prop}\label{prop:periodic-core}
	Let $\rho \subset \T$ be any geodesic, and let $\lambda,\lambda_0$ be fixed. Then any long 
periodic subsegment of $\rho$ contains a $(\lambda,\lambda_0)$-stable periodic core and the 
decomposition given in Definition \ref{defn:periodic-decomposition} is well defined.
\end{prop}
\begin{proof}
The first part follows from the definition and the second part follows from Corollary \ref
{cor:segment-uniqueness}.
\end{proof}

\section{Tripods}\label{sec:tripods}
We fix $\lambda,\lambda_0$ as given in Section
\ref{sec:path-decompositions}. Consider three nonconstant geodesic
paths $\gamma_1, \gamma_2,\gamma_3$ in $\T$ where $\gamma_1 =
[v_0,v_1], \gamma_1 = [v_1,v_2]$, and $\gamma_3 = [v_2,v_0]$. Then the
path $\gamma_1 * \gamma_2 * \gamma_3$ is called a \define{tripod}
which we denote $\tau$. All three paths have pairwise nontrivial
intersection. It is easy to see that\[\bigcap_{i,j} (\gamma_i \cap
\gamma_j) = \{b\}\] where $b $ is a point. We call $b$ the
\define{branchpoint} of $\tau$, we call the closures of the
components of $\tau \setminus \{b\}$ \define{branches}.

All three paths have $(\lambda,\lambda_0)$-periodic
decompositions. Let \[\ol{\sigma_i} \subset \sigma_i \subset \gamma_i;
\ol{\sigma_j} \subset \sigma_j \subset\gamma_j\] where the $\ol
{\sigma_m}$ are $(\lambda,\lambda_0)$-stable periodic cores and
$\sigma_m$ are maximal normal $(\lambda,\lambda_0)$-periodic segments
for $m=i,j$ and where $i,j \in \{1,2,3\}; i \neq j$. We write
$\ol{\sigma_i} \sim \ol{\sigma_j}$ if the corresponding maximal
periodic segments $\sigma_i,\sigma_j$ that contain them lie in the
same translate $\lambda'$ of $\lambda$. 

\begin{lem}\label{lem:match-up}
  Let $\tau$ be tripod as described above. If $\ol{\sigma_i} \subset
  \gamma_i, \ol{\sigma_j} \subset \gamma_j$ are
  $(\lambda,\lambda_0)$-stable periodic cores then \[\ol{\sigma_i}
  \cap \ol {\sigma_j} \neq \emptyset \Rightarrow \ol{\sigma_i} \sim
  \ol{\sigma_j}\]
\end{lem}
\begin{proof}
  There are maximal periodic segments $\sigma_i,\sigma_j \subset
  \gamma_i,\gamma_j$ resp.  such that each component of $\sigma_l
  \setminus \ol{\sigma_l}; l = i,j$ has $\lambda$-periodicity exactly
  $\kappa+4$ (by definition of stable periodic core). Let $m \in
  \ol{\sigma_i} \cap \ol{\sigma_j} \neq \emptyset$. Consider the four
  ``half segments'' \[ \sigma_i \setminus \{m\}, \sigma_j \setminus
  \{m\}\] They all have $ \lambda $-periodicity more than $\kappa+4$
  moreover it is easy to see that we can pick two of them so that they
  lie in $\gamma_i \cap \gamma_j$, which is a branch. So we have $m
  \in \gamma_i\cap \gamma_j$ and we have two segments
  $\sigma_i'\subset \sigma_i$ and $\sigma_j'\subset \sigma_j$ that
  start at $m$ and remain in $\gamma_i\cap\gamma_j$, moreover these
  segments can be chosen to be coinitial. Since they have length more
  than $\kappa+4$ by Corollary \ref{cor:segment-uniqueness} they must
  lie in the same periodic line $\lambda' = h\cdot \lambda$. The
  result now follows.
  \end{proof}

\begin{cor}\label{cor:branch-point}
  There is at most one $(\lambda,\lambda_0)$-stable periodic core
  lying in $\gamma_i \cup \gamma_j \cup \gamma_k$ whose interior
  contains the branchpoint $b$.
\end{cor}

We'll say that this periodic core is \define{non-linear}.

\begin{cor}\label{cor:domination}
  If $\ol{\sigma_i} \subset \gamma_i, \ol{\sigma_j}\subset \gamma_j$
  are $(\lambda, \lambda_0)$-stable periodic cores and $\ol{\sigma_i}
  \cap \ol{\sigma_j} \neq \emptyset$ then either $\ol {\sigma_i}$ contains
  $\ol{\sigma_j}$ or vice versa. Moreover as segments in $T$ they
  share a common endpoint.
\end{cor}
\begin{proof}
	By Lemma \ref{lem:matchup} we have $\ol{\sigma_i} \sim \ol{\sigma_j}$ and they're not disjoint. By Corollary \ref{cor:branch-point} we may assume w.l.o.g. that $\ol{\sigma_i}$ is contained in a branch. If we orient the branch 
so that it terminates at the branchpoint we have that $\ol{\sigma_i}$ and $\ol{\sigma_j}$ are 
coinitial. The result now follows immediately.
\end{proof}

\begin{defn}
  Let $\ol{\sigma_i} \subset \gamma_i$ be an
  $(\lambda,\lambda_0)$-stable periodic core. Then any other periodic
  core $\rho$ in $\gamma_{j}, j \neq i$ such that $\rho \sim
  \ol{\sigma_i}$ is called a \define{dual} of $\ol{\sigma_i}$ in
  $\gamma_j$. \end{defn}

\begin{defn}
  Let $\ol{\sigma_i} \subset \gamma_i, \ol{\sigma_m}\subset \gamma_m$
  be $(\lambda, \lambda_0)$-stable periodic cores s.t. $\ol{\sigma_i}
  \subsetneq \ol{\sigma_m}$. Then we say $ \ol{\sigma_m}$
  \define{dominates} $\ol{\sigma_i}$.
\end{defn}

\begin{prop}\label{prop:domination-bound}
  Let everything be as above and let $\ol{\sigma_m}\subset \gamma_m$ be a
  $(\lambda,\lambda_0)$-stable periodic core. Then we have the
  following: \begin{itemize}
  \item If $\ol{\sigma_m}$ has no duals then it must be contained in
    the subinterval of $ \gamma_m$ of radius $(\kappa+4)|\lambda_0|$
    centered at $b$, so \[\per{l}{\ol{\sigma_m}} \leq 2(\kappa +4).
    \]
  \item If $\ol{\sigma_m}$ has one dual $\ol{\sigma_i}$
    then \[\per{l}{\ol{\sigma_m}} \leq \per{l}{\ol{\sigma{i}}} +
    2(\kappa+4).\]
  \item If $\ol{\sigma_m}$ has two duals $\ol{\sigma_i},
    \ol{\sigma_j}$ then $\ol{\sigma_m}$ is non-linear and it dominates
    them both, moreover \[\per{l}{\ol{\sigma_m}} = \per{l}{\ol
      {\sigma_i}}+\per{l}{\ol{\sigma_j}} + 2(\kappa+4).\]
	\end{itemize}
\end{prop}

Instead of giving a proof, we provide Figure \ref{fig:domination},
from which all three points should be obvious, given the definitions
of $(\lambda,\lambda_0)$-stable periodic cores.

\begin{figure}
\centering
\psfrag{k+3}{$\kappa+4$}
\psfrag{k+6}{$2(\kappa+4)$}
\psfrag{osi}{$\ol{\sigma_i}$}
\psfrag{osj}{$\ol{\sigma_j}$}
\psfrag{osm}{$\ol{\sigma_m}$}
\psfrag{sm}{$\sigma_m$}
\insertimage{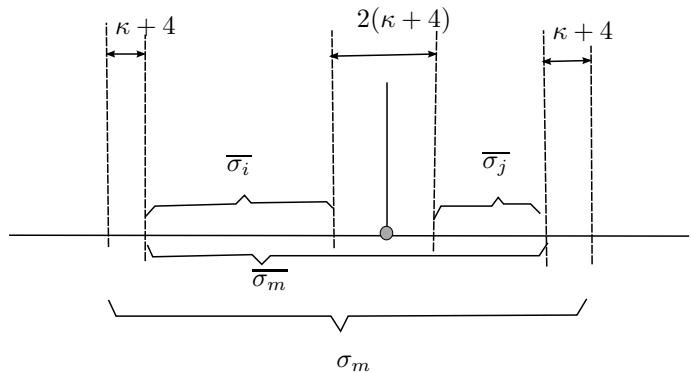}{0.75}
\caption{A proof of the third item of Proposition \ref{prop:domination-bound}. Here $\sigma_m$ 
is the maximal $(\lambda,\lambda_0)$-normal periodic subsegment of $\gamma_m$ containing $\ol
{\sigma_m}$.}
\label{fig:domination}
\end{figure}

\section{$p$-stable reduced forms}\label{sec:p-stable-forms}
Since we are dealing with one-edged splittings, it is convenient to
work combinatorially with reduced and normal forms. Consider the word
in $A*_CB$ \[w = a_1 b_1 a_2 b_2 \cdots a_n b_n; a_i \in A, b_i\in
B.\] The factors $a_i,b_j$ are called syllables. We say that
\define{$w$ is in reduced form} if none of the syllables, except maybe
$a_1$ or $b_n$ lie in $C$, and if some syllable lies in $C$ then we must have
$n=1$ and $b_1 = 1$.

Analogously consider the word in the HNN extension $\bk{A,t \mid t^\mo
  c_i t = d_i; i I}$ \[ u = a_it^{n_1}\cdots a_mt^{n_m}; a_i \in
A, n_j \in \Z \] here the factors $a_i, t^{n_j}$ are called syllables
and we say that \define{$u$ is in reduced form} if all the syllables,
except $a_1$ or $t^{n_m}$ are non trivial and there are no subwords of
the form $t^n d t^{-m}$ or $t^{-n} c t^m$, where $d \in \bk{d_i}, c\in
\bk{c_i}$ and $n,m \in \Z_{> 1}$.

We refer the reader to Chapter IV of \cite{Lyndon-Schupp-1977} for
further details on normal forms and to \cite{Serre-arbres} for further
details on the connection between free products with amalgamation, HNN
extensions and group actions on trees.

Now let $\lambda$ be some periodic line in $\T$. By replacing $\lambda$ by some
translate, if necessary, we may assume that $v_0 \in \lambda$, moreover we
can assume that $\lambda$ has a period $p$ which can be written in reduced
form \be\label{eqn:period} p= a_1d_1\cdots a_nd_n\ee where $a_i \in A$
and either
\[
\left\{ \begin{array}{ll}
d_j = t^{n_j}; n_j \in \Z  & \tr{if~} H \tr{~is an HNN extension} \\
d_j \in B  & \tr{if~}H\tr{~is an amalgam.}
\end{array}\right.
\] and all syllables are nontrivial. In this case we say $p$ is
\define{cyclically reduced}. We can finally pick a fundamental
domain $\lambda_0$ of $\lambda$ such that $v_0$ is an endpoint of
$\lambda_0$, specifically we can pick $\lambda_0 = [v_0,p\cdot v_0]$.

Let $w$ be some reduced word and let $\rho$ denote the path $[v_0,w
\cdot v_0]$. Let \[ \rho = \rho_1 * c_1 * \cdots * c_n * \rho_n\] be
its $(\lambda,\lambda_0)$- periodic decomposition. Our goal is to
rewrite $w$ so that it reflects this periodic decomposition.

\begin{defn} We say that a product $g_1g_2\cdots g_n$ is
  \define{semi-reduced} if the lengths \[ |[v_0 ,g_1g_2\cdots g_n
  v_0] | = \sum_{i=1}^n |[v_0,g_iv_0]|\] are equal. Equivalently in
  passing to a reduced form of the product, none of the syllables
  cancel completely. \end{defn}

Let $\rho = [v_0,u_0] * \ol{\sigma} * [u_1,w \cdot v_0]$ where
$\ol{\sigma} \subset h\cdot \lambda = \lambda'$ is a $(\lambda,\lambda_0)$-stable periodic
core. First note that both $u_0,u_1$
are translates of $v_0$.  Moreover we can pick $h$ so that $h\cdot v_0
= u_0.$ Figure \ref{fig:tree-diag} is provided to make the following argument more understandable.

\begin{figure}
\centering
\psfrag{v0}{$v_0$}
\psfrag{u0}{$u_0$}
\psfrag{u1}{$u_1$}
\psfrag{wv0}{$w\cdot v_0$}
\psfrag{s}{$\sigma$}
\psfrag{l}{$\lambda'$}

\insertimage{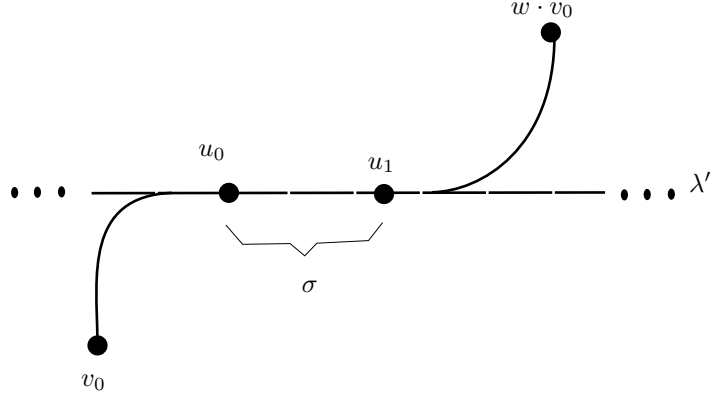}{0.75}
\caption{An $\lambda_0$-stable periodic core}
\label{fig:tree-diag}
\end{figure}

It follows by considering how $\H$ acts on $\T$ that we can divide the
word $w$ as $w = w_1 w_2 w_3$ where the $w_i$ are subwords of $w$ and
$w_1\cdot v_0 = u_0$. From this it follows that $h = w_1 a_1$ for some
$a_1 \in A$. Now since $h p h^\mo$ is a period of $\lambda'$ we have that
\[(w_1 a_1 p^m) \cdot v_0 = u_1\] where $|m| = \per{l}{\sigma}$. Finally
it is easy to see that we can pick a the coterminal subword $w_3$ so
that for some $a_2 \in A$ we have \[(w_1 a_1 p^m a_2 w_3) \cdot v_0 =
w \cdot v_0.\]

Now looking at the path\[ [v_0, (w_1 a_1) \cdot v_0] * (w_1
a_1)\cdot[v_0,p^m\cdot v_0] * (w_1 a_1 p^m a_2)\cdot[v_0,w_3 \cdot
v_0]\] we see that the product $(w_1a_1)(p^m)(a_2w_3)$ is semi-reduced
and that $(w_1a_1)$ as a reduced word can be arranged to coincide with
$w_1$ except for maybe the last syllable and that $a_2 w_3$ as a
reduced word can be arranged to coincide with $w_3$ except maybe for
the first syllable. So we were able to rewrite $w = w'_1 p^m w'_2$ as
a semi-reduced product. Now note that we can repeat this process for
the $w_1'$ and $w_2'$ factors. From this discussion we deduce the
following lemma:

\begin{lem}\label{lem:periodic-normal-form}
  Let $\H$, $\lambda, \lambda_0$ and let $p$ be cyclically
  reduced. Let $w$ be a reduced word in $\H$ and let $\rho = \rho_1 *
  c_1 * \cdots * c_n *\rho_{n+1}$ be a $(\lambda,\lambda_0)$-periodic
  decomposition of the path $[v_0,w\cdot v_0]$.  Then we can rewrite
  $w$ as a semi-reduced product
  \[w = w_1 p^{m_1} w_2 p^{m_2} \cdots p^{m_n} w_{n+1}\] which has the
  following properties:
  \begin{enumerate}
  \item $(w_1p^{m_1}\cdots p^{m_{r-1}}w_r)\cdot \lambda \supset c_r$ for all
    $r\leq n$.
  \item $|m_i| = \per{\lambda}{c_i}$ and all the factors are non trivial
      and in reduced form. 
      \end{enumerate}
  It follows that the tuple of integers $(m_1,\ldots,m_n)$ is well-defined for
  $w$.
\end{lem}

\begin{conv}
  Abusing notation, we shall denote $(|m_1|,\ldots,|m_n|)$ by
  $(m_1,\ldots,m_n)$.
\end{conv}

\begin{defn}\label{defn:periodic-normal-forms}
  Let $w \in \H$ be some reduced word and let some $p \in \H$ be
  cyclically reduced. The rewriting $w = w_1 p^{m_1} w_2 p^{m_2}
  \cdots p^{m_n} w_{n+1}$ as given in Lemma
  \ref{lem:periodic-normal-form} is called a \define{$p$-stable
    reduced form}. The subwords $p^{m_i}$ are called \define{
    $p$-stable occurrences} and the entries of tuple of positive
  integers $(m_1,\ldots,m_n)$ are called the \define{$p$-stable
    exponents} of $w$.
\end{defn}

\section{$p$-bands in Ol'shanskii diagrams}\label{sec:cutting}

For this section we fix a nontrivial homomorphism
$\phi: G \rightarrow \H$. We also fix a cyclically reduced period \[p =
a_1d_1 \cdots a_n d_n \in H\] i.e. one that satisfies the requirements
of (\ref{eqn:period}) as given in Section
\ref{sec:p-stable-forms}. Equivalently $p$ as a word in the generators
corresponds to an edge path in $\twocomplx{H}$, and in particular as
an edge path its first edge lies in $\twocomplx{A}\subset
\twocomplx{H}$. Specifically  $p$, as an edge-path has the
form: \begin{eqnarray}
  a_1 e b_2 e^\mo \cdots b^n e^\mo ;& & a_i \subset \twocomplx{A}, b_i \subset
  \twocomplx{B} \label{eqn:amalgam-path}\\
  a_1 e^{n_1} \cdots a_m e^{n_m} ; & & a_i \subset \twocomplx{A}, n_i \in \Z \label{eqn:hnn-path}
\end{eqnarray}
where the $a_i,b_i$ are understood to be edge paths and where (\ref{eqn:amalgam-path}) is a 
period where $\H$ is a free product with amalgamation and (\ref{eqn:hnn-path}) is a period 
where 
$\H$ splits as an HNN extension.

Let $\bk{Y \mid S}$ be a finite presentation of $G$. W.l.o.g. we may
assume, adding finitely more elements to $Y$ and $S$ if necessary, that the
presentation is triangular, i.e. that all the relations in $S$ have
length 3. We may also assume that no generator occurs twice in a
relator so that each 2-cell is an almost embedded triangle in
$\twocomplx{G}$ (the three corners get identified).

Let $\twocomplx{G}$ be the presentation 2-complex for $\bk{Y \mid
  S}$. The homomorphism $\phi$ enables us to construct combinatorial
maps from $\twocomplx{G}$ to $\twocomplx{H}$. We can ensure that this
combinatorial map is such that for each $y_i \in Y$, $\phi(y_i)$ is
sent to an edge path (equivalently a word) in $p$-stable reduced
form. Since $p$-stable reduced forms are semi-reduced products, it
follows this map from $\twocomplx{G}^{(1)}$ to $\twocomplx{H}^{(1)}$
induces a labeling of $\twocomplx{G}^{(1)}$ by edge paths in
$\twocomplx{H}$ that can extend to an $E$- tight Ol'shanskii diagram
$\odiag{\phi}$.

Consider the refinement $\refine{G}{\phi}$. Let $y_i$ be an edge in
$\twocomplx{G}^{(1)}$, along $y_i$ we can read in $\refine{G}{\phi} $
the edge path in $p$-stable reduced form in $\twocomplx{H}$
corresponding to our choice $p$-stable reduced form for $\phi(y_i)$.

\begin{defn}
  We call a combinatorial subarc $\sigma$ of $y_i \cap
  \refine{G}{\phi}$ \define{a $p$-stable subarc} if along it we read a
  $p$-stable occurrence of $\phi(y_i)$.
\end{defn}

This next result follows immediately from Lemma
\ref{lem:periodic-normal-form}, the definition of the resolution
$\res{\phi}:(\wt{\twocomplx{G}},\wt{y_0}) \rightarrow (\T,v_0)$ and
the fact that edges of $\twocomplx{G}$ are mapped via $\odiag{\phi}$
to $p$-stable reduced paths.

\begin{lem}\label{lem:matchup}
  Let everything be as above. Let $y_i \in Y$ and consider the edge
  $f$ in the universal cover $(\wt{\twocomplx{G}},\wt{y_0})$ labeled
  $y_i$ going from $\wt{y_0}$ to $y_i \cdot
  \wt{y_0}$. Then \begin{enumerate}
    \item The map $\res{\phi}:(\wt{\twocomplx{G}},\wt{y_0})
      \rightarrow (\T,v_0)$ sends the edge $f$ to the geodesic
      $[v_0,{\phi(y_i)\cdot v_0}].$
    \item $\res{\phi}$ maps a $p$-stable subarc of $f$ corresponding
      to a $p$-stable occurrences in $\phi(y_i)$ onto the corresponding
      $(\lambda,\lambda_0)$-stable periodic core of $[v_0,\phi(y_i)\cdot v_0]$.
  \end{enumerate}
\end{lem} 

Let $y_iy_jy_k \in S$ be an arbitrary relation where $y_i,y_j,y_k \in
Y$ and let $D_{ijk} \subset \twocomplx{G}^{(2)}$ be the corresponding
2-cell. The set \[\underbrace{[v_0,\phi(y_i)\cdot v_0]}_{\gamma_i}
\cup \underbrace{\phi(y_i) \cdot [v_0,\phi(y_j)\cdot v_0]}_{\gamma_j}
\cup \underbrace{\phi(y_iy_j) \cdot [v_0,\phi(y_k)\cdot
    v_0]}_{\gamma_k} \] forms a tripod $\gamma_i * \gamma_j * \gamma_k
= \tau$ in $\T$. We use can now the correspondence given in 2. of Lemma
\ref{lem:matchup} to associate the $p$-stable subarcs of $y_i,y_j,y_k$
to the $(\lambda,\lambda_0)$-stable periodic cores of $\gamma_i, \gamma_j,
\gamma_k$, respectively. 

\begin{defn}Let $y_r,y_w$ lie in some 2-cell $D_{ijk}$. We say that
  $p$-stable subarcs $\rho_r \subset y_r$ and $\rho_s \subset y_s$ are
  \define{dual} if the corresponding $(\lambda,\lambda_0)$-stable
  cores of $\gamma_r$ and $\gamma_s$ respectively are dual for $r,s
  \in \{i,j,k\}$. Similarly we say $\rho_r$ \define{dominates}
  $\rho_s$ if the periodic core corresponding to $\rho_r$ dominates the periodic core corresponding to $\rho_s$ in $\tau$.
\end{defn}

\begin{lem}\label{lem:p-bands}
  Let everything be as above. And let $\rho_r \subset y_r$ and $\rho_s
  \subset y_s$ be dual $p$-stable subarcs for $r,s \in \{i,j,k\}$ and
  suppose $|\rho_s| \leq |\rho_r|$ (in the sense that the $p$-stable exponent corresponding to $\rho_r$ is greater that the $p$-stable exponent correpsonding to $\rho_s$.) Then there are edge paths $\rho^+,
  \rho^-: [0,1]\rightarrow D_{ijk}$ in the refinement \[D' =
  \refine{G}{\phi} \cap D_{ijk}\] such that
  \begin{itemize}
  	\item $\rho^-$ starts and ends at endpoints of $\rho_r$ and $\rho_s$
	\item $\rho^+$ starts at the other endpoint of $\rho_r$ and ends inside $\rho_s$.
	\item Both $\rho^+,\rho^-$ do not intersect the interior of an
          $E$-cell.
	\item The edge path inside $\rho_r$ between $\rho^-(1)$ and
          $\rho^+(1)$ has the same label as $\rho_s$. 
	\end{itemize}
\end{lem}

\begin{proof}
  We pass to the universal cover $(\wt{\twocomplx{G}},\wt{y_0})$ and
  take the lift $\wt{D'}$ of $D'$ whose boundary maps via $\res{phi}$
  onto the tripod $\tau$. It is clear that since the restriction of
  the covering map $\wt{\twocomplx{G}}\rightarrow \twocomplx{G}$ is
  injective on $\wt{D'}$ minus its corners, and is compatible
  with the refinement $\refine{G}{\phi}$, it is sufficient to find paths
  $\rho^+,\rho^-$ with the desired properties in the universal cover.
	
  By Lemma \ref{lem:matchup} and Corollary \ref{cor:domination} we
  have that, via $\res{\phi}$, the lifts $\wt{\rho_r}$ and
  $\wt{\rho_s}$ are mapped to $(\lambda,\lambda_0)$-stable periodic
  cores $\sigma_r$ and $\sigma_s$ in $\T$ resp. where $\sigma_s
  \subset \sigma_r$ and they share a common endpoint $p^-$. Now by
  Lemma \ref{lem:e-tight} and the assumption that $\odiag{\phi}$ is
  $E$-tight, we have that the the $\res{\phi}$-preimages of $p^-$
  correspond to a connected component of \[D'' = \tr{closure}
  \big(\wt{D'} \setminus (E-\tr{cells})\big)\] so we can find a path
  $\rho^-$ satisfying the requirements of the lemma.
	
  By hypothesis we can read $p^{m_r}$ along $\wt{\rho_r}$ and
  $p^{m_s}$ along $\wt{\rho_s}$ with $|m_s| \leq |m_r|$. Consider the
  point $p^+$ in $p_r$ that is at the end of the path labeled
  $p^{m_s}$. It is clear that $p^+$ and the other endpoint of $p_s$
  should map to the same vertex of $\T$ so again by Lemma
  \ref{lem:e-tight} they must lie in the same component of $D''$. Thus
  there are paths $\rho^+,\rho^-$ satisfying the requirements of the
  lemma.
\end{proof}

\begin{defn}
  For any pair of dual $p$-stable subarcs $\rho_s, \rho_r$ of
  $y_r,y_s$ respectively we define the corresponding \define{$p$-band}
  to be the combinatorial disk in $D_{ijk}\cap \refine{G}{\phi}$ that
  is bounded by the four paths $\rho_s, \rho_r$ and $\rho^+,\rho^-$ as
  given in Lemma \ref{lem:p-bands}. The arcs $\rho_s,\rho_r$ are
  called \define{bases}.
\end{defn}

See Figure \ref{fig:p-bands} for an illustration.

\begin{figure}
\psfrag{y1}{$y_i$}
\psfrag{y2}{$y_j$}
\psfrag{y3}{$y_k$}
\psfrag{p1}{$\gamma_i$}
\psfrag{p2}{$\gamma_j$}
\psfrag{p3}{$\gamma_k$}
\psfrag{v0}{$v_0$}
\psfrag{v0}{$v_0$}
\psfrag{yvo}{$\phi(y_i)\cdot v_0$}
\psfrag{yyvo}{$\phi(y_iy_j)\cdot v_0$}
\centering
\insertimage{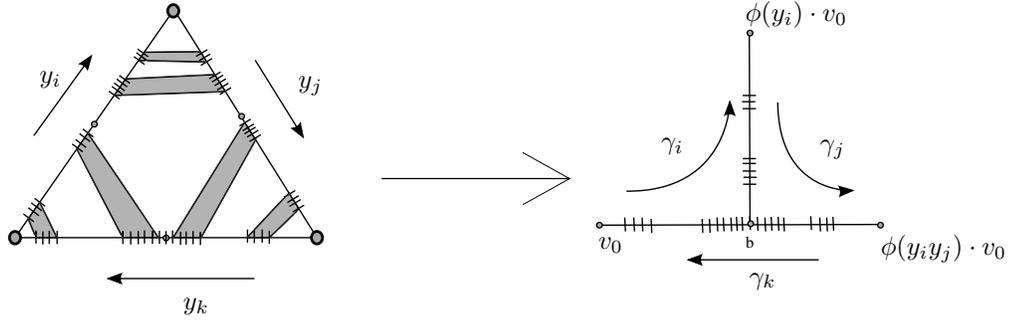}{1}
\caption{On the left, hashed segments represent $p$-stable subsegments and the grey rectangles 
are the $p$-bands. On the right the tripod $\tau$ and hashed segments representing the  
corresponding $(\lambda,\lambda_0)$-stable periodic cores} 
\label{fig:p-bands}
\end{figure}

\begin{lem}\label{lem:1-sides}
  Suppose that a band has bases along which we read $p^m$ with $m >
  \kappa$. Then the elements of $\H$ read along $\rho^+$ and $\rho^-$
  must be trivial.
\end{lem}
\begin{proof}
  Suppose along $\rho^+$ we read $c^+$ and along $\rho^-$ we read
  $c^-$. The fact that we the path $c^+p^m (c^-)^\mo p^{-m}$ gives a
  closed loop in a Van-Kampen diagram over $\twocomplx{H}$ implies the
  equality\[ p^{-m}c^+p^m =c^-.\] By Lemma \ref{lem:p-bands} and by
  choice of $p$ (see (\ref{eqn:amalgam-path}) and
  (\ref{eqn:hnn-path})) the paths $\rho^+$ and $\rho^-$ are mapped to loops in
  $\twocomplx{A} \subset \twocomplx{H}$. It follows that the elements
  $c^+,c^-$ fix the vertex $v_0$ of $\T$. We also have that
  $p^{-m}c^+p^m$ fixes $p^{-m} \cdot v_0$ in $\T$, which means that
  $c^-$ fixes the whole segment $\sigma = [v_0,p^{-m} \cdot v_0]$ of
  $\T$. Since $|\sigma|>\kappa$ the assumption that $c^-$ is not
  trivial contradicts $\kappa$-acylindricity.
\end{proof}

\begin{defn}
  A $p$-band such that along its bases we can read $p^m$ with $m \geq
  \kappa$ is called \define{long}.
\end{defn}

This next result is obvious from the previous definition and Lemma
\ref{lem:1-sides}.

\begin{cor}\label{cor:long-bands}
  Adding singular cells if necessary while preserving preserving
  $E$-tightness, we can arrange so that a long band along whose bases
  we can read $p^m$ can be be obtained by glueing $m$ discs along
  whose boundary we can read \[p * 1 * p^\mo *1\] by their edges
  labeled $1$.
\end{cor}

We are now in a position to use Bulitko's original idea and to prove
the main Theorem.

\begin{proof}[proof of Theorem \ref{thm:main}]
  Let everything be as described so far in this section. Recall that
  we want to globally bound the $\lambda$-periodicity of every $\lambda$-periodic
  subsegment of each arc $[v_0,\phi(y_i)v_0]$ where $y_i \in Y$. \emph{By Proposition \ref{prop:periodic-core} it is enough to bound the $\lambda$-periodicities of the $(\lambda,\lambda_0)$-stable periodic cores.} Note that we will not bound the \emph{number} of $\lambda$-periodic
  subsegment.

  To each $y_i \in Y$ we have the tuple $(m^i_1,\ldots,m^i_{n_i})$ of
  $p$-stable exponents corresponding to the $p$-stable reduced form
  for $\phi(y_i)$. For a fixed $p$ these exponents are well
  defined. Each 2-cell $D_{ijk} \subset \twocomplx{G}$ (along whose
  boundary we read $y_iy_jy_k$) induces a system of linear equations
  $(\star)_{ijk}$ over the positive integers with positive integral
  coefficients and constants. The variables are the entries in the
  tuple
  \[{\bf x}_{ijk} = (x^{i}_1,\ldots,x^i_{n_i}, x^{j}_1,\ldots,x^j_{n_j},
  x^k_1,\ldots,x^k_{n_k})\] corresponding to the tuple 
  \[{\bf m}_{ijk} = (m^{i}_1,\ldots,m^i_{n_i}, m^{j}_1,\ldots,m^j_{n_j},
  m^k_1,\ldots,m^k_{n_k}.)\] of $p$-stable exponents. We will now
  describe this system of equations.

  Denote by $\rho^r_l$ the $p-$stable subsegment in $y_r$ corresponding
  to the $p$-stable exponent $m^r_l$. Consider moreover the following
  system of equalities:\begin{itemize}
  \item[(0)] if $\rho^r_l$ and $\rho^s_{l'}$ are dual but neither of
    them is dominant then we have \[m^r_l = m^s_{l'}.\]
  \item[(1)] if $\rho^r_l$ has a unique dual $\rho^s_{l'}$ and it is
    dominant then we have \[m^r_l = m^s_{l'} + C^{rs}_{ll'}\]
  \item[(2)] if $\rho^r_l$ has duals $\rho^s_{l'}$ and $\rho^q_{l'}$ then we
    have \[m^r_l = m^s_{l'} + m^q_{l''} + C^{rsq}_{ll'l''}\]
  \item[(3)] if $\rho^r_l$ has no duals then we have \[m^r_l = C^r_l\]
  \item[(4)] if $m^r_l < \kappa$ then we write \[m^r_l = C^r_l\]
\end{itemize}

We call the terms $C^*_*$ \define{constants}. By Proposition
\ref{prop:domination-bound} and by (4), the constants in these
equations are all bounded by $2(\kappa+4)$. If we replace each
occurrence of $m^r_l$ by the corresponding $x^r_l$ in this system of
equalities then we get a system of equations $(\star)_{ijk}$ with
unknowns from ${\bf x}_{ijk}$. Moreover the tuple ${\bf m}_{ijk}$ is a
solution to this system of equations.

Let $(\star)$ be the system of equations consisting of all the
$(\star)_{ijk}$. We define an equivalence relation $\sim$ on the set
of $p$-stable subarcs of $\twocomplx{G}^{(1)}\cap\refine{G}{\phi}$ and
variables of $(\star)$ as follows:\begin{itemize}
  \item if $\rho_l^r$ and $\rho_l'^s$ lie in the boundary of a common 2-cell
    $D_{ijk}$ and give an equation of type (0), then we write
    $\rho_l^r \sim \rho_l'^s$.
  \item extend $\sim$ to an equivalence relation.
  \item each variable $x_l^r$ corresponds a unique $p$-stable subarc
    $\rho_l^r$ so we define $\sim$ on the set of variables analogously.
  \end{itemize}

  It is first clear that we can obtain another system of equations
  with fewer variables by replacing in $(\star)$ each variable by its
  $\sim$-class and removing redundant equations. Also if a variable
  occurs in an equation of type (4) then we can simply replace its
  whole $\sim$-class with the appropriate constant. We call this new
  system of equations $(\dagger)$.

  We already saw that the homomorphism $\phi:G\rightarrow \H$, and more
  particularly $\odiag{\phi}$ gave a solution ${\bf m}$ to
  $(\dagger)$. Suppose we were given another positive integral
  solution ${\bf m}''$to $(\dagger)$. First note that by equations of
  type (4), we may assume that all variables in $(\dagger)$ correspond
  to \emph{long} $p$-stable subarcs in
  $\twocomplx{G}^{(1)}\cap\refine{G}{\phi}$. This means that Corollary
  \ref{cor:long-bands} applies. Since ${\bf m}''$ is a solution to
  $(\dagger)$ it means that in each 2-cell of $\twocomplx{G}$ we can
  either surger in or cut out pieces of $p$-bands as in Figure
  \ref{fig:p-band-surgery}.
  
  \begin{figure}
    \psfrag{1}{$1$}
    \psfrag{p}{$p$} 
    \psfrag{p3}{$p^3$} 
    \psfrag{p4}{$p^4$} 
    \psfrag{p6}{$p^6$} 
    \psfrag{a7}{$\alpha_7$}
    \psfrag{a1}{$\alpha_1$}
    \psfrag{a2}{$\alpha_2$}
    \psfrag{a3}{$\alpha_3$}
    \psfrag{a4}{$\alpha_4$}
    \psfrag{a5}{$\alpha_5$}
    \psfrag{a6}{$\alpha_6$}
    \centering
    \insertimage{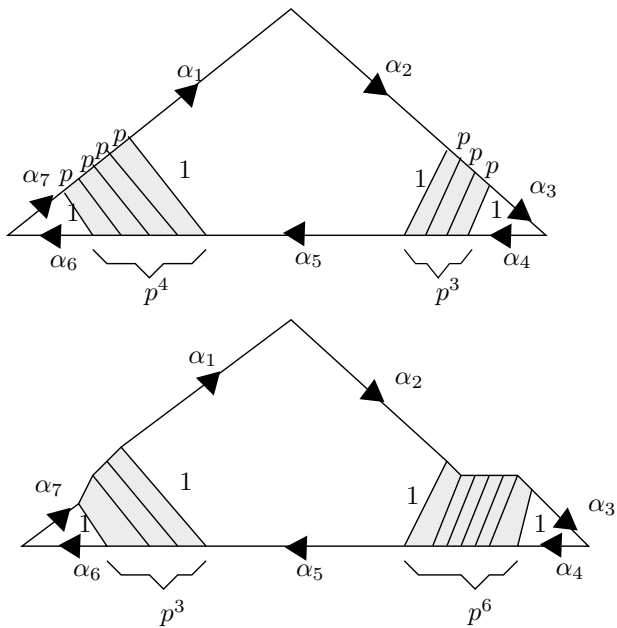}{0.66}
    \caption{Doing surgery on $p$-bands. The labels of the arcs
      $\alpha_i$ remains unchanged and this is still a Van-Kampen diagram}
    \label{fig:p-band-surgery}
  \end{figure}

  Note that this gives another Ol'shanskii diagram of $\twocomplx{G}$
  over $\twocomplx{H}$. So by Lemma \ref{lem:o-diags} there is a
  corresponding homomorphism $\phi'':G\rightarrow \H$ realizing this
  new Ol'shanskii diagram $\odiag{\phi''}$. $\phi''$ satisfies all the
  requirements of the $\phi'$ given in Theorem \ref{thm:main} except
  for the bound on $\lambda$-periodicities. Note however that each $y _i \in
  Y$ is still sent to a path in $p$-stable reduced form, but that the
  $\lambda$-periodicity of the long $p$-stable occurrences are now given by the
  entries in the new solution ${\bf m}''$. \emph{It therefore follows
    that all that remains is to find another solution ${\bf m}'$ to
    $(\dagger)$ where the size of the largest entry is controlled.}
  
  We first note that $(\dagger)$ decomposes into a union of
  independent subsystems of equations which divide into two types:
  \begin{itemize}
  \item[(a)] $x = x$
  \item[(b)] A system of equations containing an equation of type (2)
    or (3)
  \end{itemize}
  
  There is no way to bound the number of equations of type (a),
  however we immediately see that $x \mapsto 1$ is a positive integral
  solution to these equations. Let $(\dagger)'$ be the system of
  equations consisting of all subsystems of equations of type (b). By
  Corollary \ref{cor:branch-point} hand for each 2-cell $D_{ijk}$
  contributes at most one equation of type (2) and contributes at most
  three equations of type (1). Hence the number of equations in
  $(\dagger)'$ is bounded above by $4|S|$ (assuming $S$ is
  triangular). Moreover from the definition of the $(\star)_{ijk}$
  each equation consists of at most four terms which are either
  variables or constants and the coefficients are all 1. Moreover all
  the constants are bounded above by $2(\kappa + 4)$.
  
  We now show how to compute $\func:\N \times \N \times \N \rightarrow
  \N$ as given in Theorem \ref{thm:main}. Up to choice of alphabet
  there is a finite number $N(g,r)$ of finite presentations with $g$
  generators and $r$ relations. This gives us a finite set \[\{\bk{Y_i
    \mid S_i } \mid i = 1, \ldots N(g,r)\}\] of corresponding
  \emph{triangular} presentations (triangulating a presentation can be done algorithmically.) Let \[K(g,r) = \max_{i=1\ldots
    N(g,r)}(|S_i|)\] Then there is a finite set $\alleqns{g}{r}{\kappa}$
  consisting of all systems of equations such that:
  \begin{itemize}
  \item There are at most $4K(g,r)$ equations.
  \item Each equation has at most 4 terms and all coefficients are 1 so each term is just  variable or a constant.
  \item Each constant is positive and bounded above by $2(\kappa+4).$
   \end{itemize}
  
  For each system of equations $(\star)_j \in
  \alleqns{g}{r}{\kappa}$ that has a positive integral solution,
  compute such a solution ${\bf v_j}$ (this can be done, it's just linear programming.) Among all the entries in all
  the ${\bf v_j}$ we take $\func'(g,r,\kappa)$ to be the
  maximum. We let $\func(g,r,\kappa) = \func'(g,r,\kappa) + 2(\kappa + 4)+1$. This $\func$ is clearly computable and satisfies the requirements of Theorem \ref{thm:main}.
\end{proof}

\bibliographystyle{alpha}\bibliography{biblio.bib}

\end{document}